\documentclass[11pt]{amsart}
\usepackage{amsfonts}
\usepackage{hyperref}
\usepackage[utf8]{inputenc}
\usepackage[T1]{fontenc}
\usepackage[dvips]{graphicx}
\usepackage{amssymb}
\usepackage{amsmath}
\usepackage{dsfont}
\usepackage{color}
\usepackage{latexsym}
\usepackage{bbm}
\usepackage{color}
\usepackage{amsthm}
\usepackage{multicol}
\usepackage[top   = 2.75cm,
bottom = 2.50cm,
left   = 2.50cm,
right  = 2.00cm]{geometry}

\bibstyle{plain}
\theoremstyle{plain}
\newtheorem{theorem}{Theorem}
\newtheorem{lemma}[theorem]{Lemma}

\newtheorem{corollary}[theorem]{Corollary}

\theoremstyle{remark}

\newtheorem{remark}[theorem]{Remark}

\begin{document}

\title[]{An analogue of the Gauss-Kuzmin problem
 for continued $A_{2}$-fractions
}
\author[M.\ V.\ Pratsiovytyi, O.\ P.\ Makarchuk and 
S.\ P.\ Ratushniak]{ M.\ V.\ Pratsiovytyi, O.\ P.\ Makarchuk and S.\ P.\ Ratushniak}

\newcommand{\eacr}{\newline\indent}

\address{M.V. Pratsiovytyi\eacr
Institute of Mathematics of NASU,
 Dragomanov Ukrainian State University, Kyiv,
Ukraine\acr ORCID 0000-0001-6130-9413}
\email{prats4444@gmail.com}

\address{O.P. Makarchuk\eacr
Volodymyr Vynnychenko Central Ukrainian State University,
Institute of Mathematics of NASU
\acr ORCID 0000-0002-1001-8568}
\email{makarchuk@imath.kiev.ua}

\address{S.P. Ratushniak\eacr
Institute of Mathematics of NASU,
 Dragomanov Ukrainian State University, Kyiv,
Ukraine\acr ORCID 0009-0005-2849-6233}
\email{ratush404@gmail.com}

\keywords{Gauss-Kuzmin-type problem, invariant measure, continued $A_{2}$-fractions, left shift operator, representation by an $A_2$-continued fraction.\\This work was supported by a grant from the Simons
Foundation (SFI-PD-Ukraine-00014586, M.P., O.M., S.R.)}

\date{\today}

\newcommand{\acr}{\newline\indent}

\begin{abstract}
In this paper, for the chain $A_2$-representation of numbers $x\in[\frac{1}{2};1]$, that is, $x=1/u_1+1/u_2+....+1/u_n+...$, where $u_n\in \{\frac{1}{2},1\}$, a problem in metric number theory and dynamical systems is solved which is an analogue of Gauss’s problem for simple continued fractions. Specifically, we study the asymptotic behavior of the sequence of Lebesgue measures $\lambda$ of the sets
\[E_n(x)=\{u: u\in [\frac{1}{2};1], \; 1/u_{n+1}+1/u_{n+2}+1/u_{n+3}+...<x\},\]
where $u= 1/u_{1}+1/u_{2}+1/u_{3}+...+1/u_{k}+...$, $u_k\in \{\frac{1}{2};1\}$.

It is proved that $\lim\limits_{n\to\infty}\lambda[E_n(x)]=\frac{1}{2}\ln^{-1}\left(\frac{10}{9}\right)\ln\frac{5(x+1)}{3(x+2)}=\frac{1}{2}f(x)$. Estimates for the rate of convergence of the sequence $\frac{1}{2}f_n(x)=\lambda[E_n(x)]$ to the corresponding limiting function $\frac{1}{2}f(x)$ are obtained.
\end{abstract}

\maketitle

\section*{Introduction}
We consider real numbers represented as continued fractions.

Representations of real numbers by simple continued fractions (whose partial quotients are natural numbers) have numerous applications across various areas of pure and applied mathematics. This system of encoding (representation) of real numbers possesses an infinite alphabet and zero redundancy: each real number admits at most two representations, while every irrational number has a unique representation. The convergence properties yield of simple continued fractions provides highly efficient rational approximations to irrational numbers.

Representations of real numbers by continued fractions over a finite alphabet (whose elements are positive real numbers), the development of whose theory was initiated in~\cite{prats}, have also found applications in the theory of continuous locally complex functions, the theory of singular probability distributions, fractal analysis, and related fields. In the case of number representations by $A_2$-continued fractions (a two-symbol alphabet) under the assumption of zero redundancy, two important structural features arise: binary symbolic encoding and the continued-fraction nature of the expansion. These properties motivate the investigation of analogues of classical problems from metric and probabilistic number theory and ergodic theory that were originally studied for simple continued fractions. One such problem is the Gauss–Levy problem.

 Let us recall the key concepts.
An expression of the form
\begin{equation}\label{mp1}
a_0+1/a_1+1/a_2+...+1/a_n+...=[a_{0};a_{1},a_{2},...,a_{n},...]
\end{equation} is called an infinite continued fraction, $a_{n}$ is called its element (or link), while the finite continued fraction $[a_{0};a_{1},a_{2},...,a_{n}]=a_0+1/a_1+1/a_2+...+1/a_n$ is called the segment of the $n$-th rank of the continued fraction~\eqref{mp1}. If the elements of a continued fraction are natural numbers, then the continued fraction is called elementary (or simple).  Next, we will consider continued fractions in which $a_0=0$ and the terms of the sequence $(a_n)$ are positive real numbers.

A fraction $\frac{p_n}{q_n}$, which is the value of the finite continued fraction $[0;a_{1},a_{2},...,a_{n}],$ is called the $n$-th convergent of the continued fraction $[0;a_{1},a_{2},...,a_{k},...]$. The finite limit (when it exists) of the sequence $(\frac{p_n}{q_n})$ is called the value of the infinite continued fraction~\eqref{mp1}. In this case, the continued fraction is said to be convergent. A classical convergence criterion is Seidel's theorem~\cite{seid}: the continued fraction~\eqref{mp1} converges if and only if the series
$\sum\limits_{n=1}^{\infty}a_n$ diverges.

The mapping $T^n$, defined by
\begin{equation}\label{oper}
T^n([0;a_{1}, a_{2}, ..., a_{n}, ...])=[0;a_{n+1}, a_{n+2},...,a_{n+k}, ...],
\end{equation}
is called the $n$-fold left shift operator (in the space of representation of numbers). Clearly, \(T^n(x)=T(T^{n-1}(x)),\)
where $T^1(x)\equiv T(x)=\frac{1}{x}-a_1(x)$.

To ensure that the mapping $T^n$ defines a well-defined function on an interval, one must address the ambiguity at points that admit more than one continued fraction representation. This is easily resolved by adopting a convention to use only one of the existing expansions.
 
One of the earliest problems in the metric theory of simple continued fractions and the associated dynamical systems is Gauss's problem on the asymptotic behaviour of the Lebesgue measure $\lambda$ of the set $E_n(x)=\{a\in [0;1]: T^n(a)\leq x\}$.

Gauss established that $\lim\limits_{n\to\infty}\lambda (E_n(x))=\frac{\ln(1+x)}{\ln2}$, although he never published a proof. Gauss was also interested in the order of smallness of the difference $\lambda(E_n(x))-\frac{\ln(1+x)}{\ln2}$. A proof of the existence of the limit together with estimates for this difference was later obtained by R. O. Kuzmin in 1928~\cite{kuz2}. Therefore this problem is commonly known as the Gauss-Kuzmin problem for simple continued fractions~\cite{chin}. The first improvement of Kuzmin's estimates was due to P. Levy~\cite{lev}.
An analogous problem for other numerical expansions becomes both interesting and nontrivial whenever the underlying coding system is not self-similar, for example, for Engel series or the Ostrogradsky--Sierpinski--Pierce expansions~\cite{M_2023}.

In the present work we study an analogue of the Gauss--Kuzmin problem for $A_2$-continued fractions (continued fractions with the finite alphabet $A=\{\frac{1}{2},1\}$).

Recall~\cite{prats} that an infinite $A_2$-continued fraction (or simply an $A_2$-fraction) is called a continued fraction of the form
  $$0+1/\alpha_1+1/\alpha_2+...+1/\alpha_n+...=[0;\alpha_{1},\alpha_{2},...,\alpha_{n},...], \;\alpha_n\in A.$$

For almost all numbers $x\in[\frac{1}{2};1]\equiv I$ the representation by an $A_2$-continued fraction is unique, that is,
$x=[0;\alpha_1,\alpha_2,...,\alpha_n,...]=\Delta^{A_2}_{\alpha_1\alpha_2...\alpha_n...}$, where $\alpha_n=\alpha_n(x)$. Only a countable set dense in $[\frac{1}{2};1]$ admits two distinct representations:
$[0;b_{1}, ...,b_{n}, \frac{1}{2},(\frac{1}{2},1)]=[0;b_{1}, ...,b_{n}, 1,(1,\frac{1}{2})].$
Parentheses indicate periodicity. Such numbers are called $A_2$-binary. Clearly, the set of $A_2$-binary numbers may be neglected in metric problems. The symbolic expression $\Delta^{A_2}_{\alpha_1\alpha_2...\alpha_n...}$ is called the $A_2$-representation of the number $x$.

A cylinder (an $A_2$-cylinder) of rank $m$ with base  $c_1c_2...c_m$ is called the set
\[\Delta^{A_2}_{c_1c_2...c_m}=\{x: x=\Delta^{A_2}_{c_1c_2...c_m\alpha_1\alpha_2...\alpha_n...}, \alpha_n\in A\}.\]

Each cylinder of rank $m$ is the union of two cylinders of rank $(m+1)$, and cylinders of the same rank do not overlap.
The cylinder $\Delta^{A_2}_{c_1c_2...c_m}$  is a segment with endpoints: $[0;c_1,...,c_m,(\frac{1}{2},1)]$ and
$[0;c_1,...,c_m,(1,\frac{1}{2})]$, and its length is given by
\begin{eqnarray}
                       |{\Delta_{c_1c_2...c_m}^{A_2}}| = \frac{1}{(q_{m-1}+q_m)(q_{m-1}+2q_m)}.
            \end{eqnarray}
            
\section{Some Properties of the Left Shift Operator}

Let $T^{-n}(u)\equiv\{x: T^{n}(x)=u\}$ and $T^{-n}([\alpha;\beta])\equiv \{x: T^n(x)=u\in [\alpha;\beta]\}$.
Then it is easy to see that, for $A_2$-continued fractions, the set $E_n(x)\equiv \{u: T^n(u)\leq x\}$ coincides with the set $T^{-n}([0,5;x])$.

We denote  $[x;y]^{*}=[\min\{x,y\};\max\{x,y\}]$, $\eta(\cdot)=2\lambda(\cdot)$.

\begin{lemma}
  The preimage of an interval $(\alpha;\beta)$ under the map $T$ consists of two intervals $(\frac{1}{1+\beta};\frac{1}{1+\alpha})$ and  $(\frac{1}{0,5+\beta};\frac{1}{0,5+\alpha})$, which do not overlap and touch only in the: $\alpha=\frac{1}{2}; \beta=1$.
\end{lemma}
\begin{proof}
  Since $T(x)=\frac{1}{x}-a_1(x)$ with $a_1\in A$, there exist $u$, $v$ such that
  \[\alpha=\frac{1}{u}-a_1(u),\;\; \beta=\frac{1}{v}-a_1(v),\]
  hence
  $u=\frac{1}{\alpha+a_1(u)}$,
      $v=\frac{1}{\beta+a_1(v)}.$
    Thus
    $T^{-1}((\alpha;\beta))=(u_1;v_1)\cup(u_{0};v_{0}).$
    We consider difference $u_{0}-v_1=\frac{1}{\alpha+0,5}-\frac{1}{\beta+1}=\frac{\beta-\alpha+0,5}{(\alpha+0,5)(\beta+1)}\geq0$.
    Hence the two intervals $(u_1;v_1)$ and $(u_{0};v_{0})$ do not overlap, and they touch only when  $u_{0}=v_1$, i.e. only when $\alpha=0,5; \beta=1$.
\end{proof}
\begin{lemma}
 The preimage of an interval $(u,v)$ under the mapping $T^n$ consists of $2^n$ intervals
  \begin{equation}\label{seg}
    ([0;\alpha_1,...,\alpha_n,u^{-1}],[0;\alpha_1,...,\alpha_n,v^{-1}])^*,
  \end{equation} where
  $(\alpha_1,...,\alpha_n)\in A^n$. These intervals are pairwise disjoin.
\end{lemma}
\begin{proof}
It is evident that the image of each interval of the system~\eqref{seg} is the interval $(u,v)$.
  There are exactly $2^n$ such intervals. 
  Every interval of this form~\eqref{seg} is contained in cylinder $\Delta^{A_2}_{\alpha_1...\alpha_n}$ of the $n$-th rank. Since two distinct $A_2$-cylinders of the same rank may share at most one boundary point, the corresponding intervals~\eqref{seg} do not overlap.
\end{proof}
\begin{corollary}
  For every interval $[\alpha;\beta]\subset[0,5;1]$ and every $n\in N$, the following equality holds
  \begin{equation}\label{op}
    T^{-n}([\alpha;\beta])=T^{-n}([0,5;\beta])\setminus T^{-n}([0,5;\alpha)).
      \end{equation}
      \end{corollary}
Indeed, since by the definition of the set
\begin{align*} T^{-n}([\alpha;\beta])&=\{x: T^n(x)=u\in [\alpha;\beta]\}=
\{x: T^n(x)=u\in [0,5;\beta]\setminus [0,5;\alpha)\}=\\
&=\{x: T^n(x)=u\in [0,5;\beta] \wedge u\notin [0,5;\alpha)\}=\\
&=\{x: T^n(x)=u\in [0,5;\beta]\}\setminus \{x: T^{n}=u\in [0,5;\alpha)\}=\\
&=T^{-n}([0,5;\beta])\setminus T^{-n}([0,5;\alpha)).
\end{align*}
\section{Measure of the Set of Preimages of Segments}
 For any segment $[\alpha;\beta]\subseteq[0,5;1]$ we define   $F_n(\alpha;\beta)\equiv2\lambda\left(T^{-n}([\alpha;\beta])\right)$.
\begin{lemma}\label{lema1}
  For any segment $[\alpha;\beta]\subseteq I$ and any $n\in N$ 
  $$
  F_{n+1}(\alpha;\beta)=F_n \left( \frac{1}{0,5+\beta};\frac{1}{0,5+\alpha} \right)+F_n\left(\frac{1}{1+\beta};\frac{1}{1+\alpha}\right).
  $$
\end{lemma}
\begin{proof}
By the previous lemma and additivity of Lebesgue measure
    \begin{equation}\label{Form1}
    F_n(\alpha;\beta)=\sum_{(\alpha_1,...,\alpha_n)\in A^n}\eta\left([[0;\alpha_1,\alpha_2,\ldots,\alpha_n,\alpha^{-1}];
    [0;\alpha_1,\alpha_2,\ldots,\alpha_n,\beta^{-1}]]^* \right).
     \end{equation}
Replacing $(\alpha;\beta)$ in \eqref{Form1} by $\left( \frac{1}{\beta+0,5};\frac{1}{\alpha+0,5}\right)$ and $(\alpha;\beta)$ to $\left( \frac{1}{\beta+1};\frac{1}{\alpha+1}\right)$ we obtain
\begin{align*}
  F_{n+1}(\alpha;\beta)=&\sum_{(\alpha_1,...,\alpha_{n+1})\in A^{n+1}}\eta
 \left([[0;\alpha_1,\ldots,\alpha_n,\alpha_{n+1},\alpha^{-1}];
 [0;\alpha_1,\ldots,\alpha_n,\alpha_{n+1},\beta^{-1}]]^* \right)=\\
  =&\sum_{(\alpha_1,...,\alpha_n)\in A^n}\eta\left(\left
[\left[0;\alpha_1,\ldots,\alpha_n,\frac{1}{2}+\alpha\right];
\left[0;\alpha_1,\ldots,\alpha_n,\frac{1}{2}+\beta\right]\right]^* \right)+\\
+&\sum_{(\alpha_1,...,\alpha_n)\in A^n}\eta\left(\left[\left[0;\alpha_1,\ldots,\alpha_n,1+\alpha\right];\left[0;
\alpha_1,\ldots,\alpha_n,1+\beta\right]\right]^* \right)=\\
=&F_n\left(\frac{1}{0,5+\beta}; \frac{1}{0,5+\alpha}\right)+F_n\left(\frac{1}{1+\beta}; \frac{1}{1+\alpha}\right).
\qedhere
\end{align*}
 \end{proof}
\begin{corollary}
For every interval $[\alpha;\beta]\subseteq I$ and every
$n\in N$, the following holds
  $$
  F_{n}(\alpha;\beta)=F_n(0,5;\beta)-F_n(0,5;\alpha).
  $$ 
\end{corollary}
\begin{theorem}\label{lema2}
   For every segment $[\alpha;\beta]\subseteq I$ the condition is fulfilled
   $$
   \frac{\frac{32}{3}(\beta-\alpha)}{(\alpha+\frac{5}{6})(\beta+3)}\leq F_n(\alpha;\beta)\leq \frac{\frac{77}{6}(\beta-\alpha)}{(\alpha+3)(\beta+\frac{5}{6})}.
   $$
 \end{theorem}
\begin{proof}
    Consider the point $x=[0;\alpha_1,\alpha_2,...,\alpha_n]=\frac{p_n(x)}{q_n(x)}$, $\alpha_n\in A$, and the points
 $x_1=[0;\alpha_1,...,\alpha_n,\frac{1}{\alpha}]=\frac{p_{n+1}(x_1)}{q_{n+1}(x_1)}=\frac{\frac{1}{\alpha}p_n(x)+p_{n-1}(x)}
 {\frac{1}{\alpha}q_n(x)+q_{n-1}(x)}$ and
 $x_2=[0;\alpha_1,...,\alpha_n,\frac{1}{\beta}]=\frac{p_{n+1}(x_2)}{q_{n+1}(x_2)}=\frac{\frac{1}{\beta}p_n(x)+p_{n-1}(x)}
 {\frac{1}{\beta}q_n(x)+q_{n-1}(x)}$. Then
 \begin{align*}
   \eta([x_1,x_2]^*)=&2\left|\frac{p_{n+1}(x_1)}{q_{n+1}(x_1)}-
 \frac{p_{n+1}(x_2)}{q_{n+1}(x_2)}\right|=
   2\left|\frac{\frac{1}{\alpha}p_{n}(x)+p_{n-1}(x)}{\frac{1}{\alpha}q_{n}(x)+q_{n-1}(x)}-
   \frac{\frac{1}{\beta}p_{n}(x)+p_{n-1}(x)}{\frac{1}{\beta}q_{n}(x)+q_{n-1}(x)}\right|=\\
   =&2\left|\frac{(\frac{1}{\alpha}-\frac{1}{\beta})(p_n(x)q_{n-1}(x)-p_{n-1}(x)q_{n}(x))}
   {(\frac{1}{\alpha}q_{n}(x)+q_{n-1}(x))(\frac{1}{\beta}q_{n}(x)+q_{n-1}(x))}\right|=\\
   =&2\left|\frac{(\beta-\alpha)\cdot (-1)^{n+1}}
   {(q_{n}(x)+\alpha q_{n-1}(x))(q_{n}(x)+\beta q_{n-1}(x))}\right|=\\
   =&\frac{2(\beta-\alpha)}
   {(q_{n}(x)+\alpha q_{n-1}(x))( q_{n}(x)+\beta q_{n-1}(x))},
 \end{align*}
     since, as is well known,
   $p_n(x)\cdot q_{n-1}(x)-q_n(x)\cdot p_{n-1}(x)=(-1)^{n+1}.$
  It is well known~\cite{chin,PGLR_2023}, that
   $$
   \frac{q_n(x)}{q_{n-1}(x)}=
   [\alpha_n;\alpha_{n-1},\alpha_{n-2},\ldots,\alpha_1] \in[\frac{5}{6};3].
   $$

   For fixed $\gamma \in I$ in the segment $[\frac{5}{6};3]$ define
    $$h_{1}(z)=\frac{z+0,5}{z+\gamma},\;\;  h_{2}(z)=\frac{z+1}{z+\gamma}.$$
   Since
   $$h_{1}^\prime(z)=\frac{\gamma-0,5}{(z+\gamma)^2}\geq 0\;\; \forall z\in [\frac{5}{6};3]  \Rightarrow   h_{1}(\frac{q_{n}}{q_{n-1}})\in\left[\frac{ \frac{4}{3}}{\gamma+\frac{5}{6}};\frac{3,5}{\gamma+3}\right],$$
    $$h_{2}^\prime(z)=\frac{\gamma-1}{(z+\gamma)^2}\leq 0\; \forall z\in [\frac{5}{6};3] \Rightarrow   h_{2}(\frac{q_{n}}{q_{n-1}})\in\left[\frac{4}{\gamma+3};\frac{\frac{11}{6}}{\gamma+\frac{5}{6}}\right].$$
   It is clear that
    $$\sum_{(\alpha_1,...,\alpha_n)\in A^n}\frac{1}{(q_n+0,5q_{n-1})(q_n+q_{n-1})}=\sum_{(\alpha_1,...,\alpha_n)\in A^n}\eta(\Delta^{A_2}_{\alpha_1\alpha_2...\alpha_n})=1.$$
   Then we have
   $$
   \frac{1}{q_n+\gamma q_{n-1}}: \frac{1}{q_n+0,5 q_{n-1}}=\frac{0,5+\frac{q_{n}}{q_{n-1}}}{\gamma+\frac{q_{n}}{q_{n-1}}}=h_{1}\left(\frac{q_{n}}{q_{n-1}}\right)
    \in\left[\frac{ \frac{4}{3}}{\gamma+\frac{5}{6}};\frac{3,5}{\gamma+3}\right],
   $$
   $$
   \frac{1}{q_n+\gamma q_{n-1}}:\frac{1}{q_n+q_{n-1}}=\frac{1+\frac{q_{n}}{q_{n-1}}}{\gamma+\frac{q_{n}}{q_{n-1}}}=h_{2}\left(\frac{q_{n}}{q_{n-1}}\right)
   \in\left[\frac{4}{\gamma+3};\frac{\frac{11}{6}}{\gamma+\frac{5}{6}}\right].
   $$
   Where from
   \[
   F_n(\alpha;\beta)\geq 2\cdot\frac{ \frac{4}{3}(\beta-\alpha)}{\alpha+\frac{5}{6}}\cdot\sum_{\stackrel{\alpha_i\in A}{i=\overline{1,n}}}\frac{1}{(q_n+0,5q_{n-1})(q_n+q_{n-1})}\cdot\frac{4}{\beta+3}=\frac{\frac{32}{3}(\beta-\alpha)}{(\alpha+\frac{5}{6})(\beta+3)}.
\]
     Similarly
    \[
   F_n(\alpha;\beta)\leq 2\cdot\frac{3,5(\beta-\alpha)}{\alpha+3}\cdot\sum_{\stackrel{\alpha_i\in A}{i=\overline{1,n}}}\frac{1}{(q_n+0,5q_{n-1})(q_n+q_{n-1})}\cdot\frac{\frac{11}{6}}{\beta+\frac{5}{6}}=
   \frac{\frac{77}{6}(\beta-\alpha)}{(\alpha+3)(\beta+\frac{5}{6})}.
   \qedhere\]
 \end{proof}

\section{Differential Properties of the Function $f_{n}(x)=2\lambda(T^{-n}([0,5;x]))$}
Let $\eta(\cdot)=2\lambda(\cdot)$ and $f_{n}(x)=\eta(T^{-n}([0,5;x]))$.
\begin{remark}\label{rem}
By the definition of the functions $f_n(x)$ and $F_n$, for arbitrary $x\in I$ and $n\in N$, the following holds
    \[f_n(x)=\eta(T^{-n}([0,5;x]))=F_n(0,5;x).\]
 \end{remark}
 \begin{lemma}\label{lemma:basic}
  For every natural number $n>1$ and $x\in I$, the following equality holds
      \begin{equation}\label{frec}
     f_{n+1}(x)=f_n(1)-f_n\left(\frac{1}{0.5+x}\right)+f_n\left(\frac{2}{3}\right)-f_n\left(\frac{1}{1+x}\right).
   \end{equation}
 \end{lemma}
 \begin{proof}
 According to Remark~\ref{rem}, Lemma~\ref{lema1}, and its corollary, the following equalities hold
    \begin{align*}
     f_{n+1}(x)=&F_{n+1}(0,5;x)=F_{n}\left(\frac{1}{0,5+x};1\right)+F_n\left(\frac{1}{1+x};\frac{2}{3}\right)=\\
     =&F_n(0,5;1)-F_n\left(0,5;\frac{1}{0,5+x}\right)+F_n\left(0,5;\frac{2}{3}\right)-F_n\left(0,5;\frac{1}{1+x}\right)=\\
     =&f_n(1)-f_n\left(\frac{1}{0,5+x}\right)+f_n\left(\frac{2}{3}\right)-f_n\left(\frac{1}{1+x}\right).\qedhere
   \end{align*}
 \end{proof}
\begin{lemma}\label{lema3}
  For every $x\in I$ and every natural number $n$, the following inequalities hold:
  $$
  \frac{\frac{60}{7}}{(x+1)(x+2)}\leq f_n^\prime(x)\leq \frac{\frac{21}{2}}{(x+1)(x+2)}.
  $$
\end{lemma}
 \begin{proof}
Since for each $x \in I$ the equality is satisfied
$$f_{n}(x)=F_n(0,5;x)=\sum_{(\alpha_1,...,\alpha_n)\in A^n}\frac{2x-1}{(q_{n}+0,5q_{n-1})(q_{n}+xq_{n-1})}, $$
then $f_{n}^{'}(x)$ exist for every $x \in I$ and $n\in N$.

  Using Lemma \ref{lema2}, for each  $x \in I$ we have
  $$
  f_{n}^\prime(x)=\lim_{\stackrel{x\in [\alpha;\beta]}{(\beta-\alpha)\rightarrow 0}}\frac{F_n(\alpha;\beta)}{\beta-\alpha} \geq
   \lim_{\stackrel{x\in [\alpha;\beta]}{(\beta-\alpha)\rightarrow 0}}\frac{\frac{32}{3}(\beta-\alpha)}{(\alpha+\frac{5}{6})(\beta+3)(\beta-\alpha)} \geq \frac{\frac{32}{3}}{(x+\frac{5}{6})(x+3)}.
   $$
Similarly,
$$
  f_{n}^\prime(x)\leq\frac{\frac{77}{6}}{(x+\frac{5}{6})(x+3)} \;\; \forall x\in I.
  $$
Consider the function
$$r(x)=\frac{(x+1)(x+2)}{(x+\frac{5}{6})(x+3)}.$$
Since
$$r^\prime(x)=\frac{30 x^2+36x-6}{(x+\frac{5}{6})(x+3)}>0 \;\;\; \forall x \in I, $$
the function $r$ is increasing on $I$. Hence,
$$
  f_{n}^\prime(x)\geq \frac{\frac{32}{3}}{(x+\frac{5}{6})(x+3)} \geq \frac{\frac{32}{3}r(0,5)}{(x+1)(x+2)}=\frac{\frac{60}{7}}{(x+1)(x+2)}.$$
Analogously,
\[
  f_{n}^\prime(x)\leq \frac{\frac{77}{6}}{(x+\frac{5}{6})(x+3)} \leq \frac{\frac{77}{6}r(1)}{(x+1)(x+2)}=\frac{\frac{21}{2}}{(x+1)(x+2)}.\qedhere\]
\end{proof}
\begin{lemma}\label{lema4}
 For every natural number $n$ and every $x\in I$ 
 $$\left| f_n^{\prime\prime}(x)\right|\leq \frac{33}{8}.$$
\end{lemma}
\begin{proof}
  Clearly, $f_0(x)=2x-1$ for all $x\in I$. From Lemma \ref{lema3}, for each natural $n$ and each $x\in I$ 
  $$
   f_n^\prime(x)=\lim_{\stackrel{x\in [\alpha;\beta]}{(\beta-\alpha)\rightarrow 0}}\frac{1}{\beta-\alpha}\sum_{\stackrel{\alpha_i\in A}{i=\overline{1,n}}}\frac{2(\beta-\alpha)}{(q_{n}+\alpha q_{n-1})(q_{n}+\beta q_{n-1})}=
    \sum_{\stackrel{\alpha_i\in A}{i=\overline{1,n}}}\frac{2}{(q_n+xq_{n-1})^2}.
  $$
Thus,
   $$
    f_n^{\prime\prime}(x)=\sum_{(\alpha_1,...,\alpha_n)\in A^n}\frac{-4q_{n-1}}{(q_n+xq_{n-1})^3}.
  $$
  Hence,
  $$
    \left| f_n^{\prime\prime}(x)\right|\leq \sum_{(\alpha_1,...,\alpha_n)\in A^n}\frac{4q_{n-1}}{(q_n+xq_{n-1})^3}\leq \sum_{(\alpha_1,...,\alpha_n)\in A^n}\frac{4q_{n-1}}{(q_n+0,5q_{n-1})^3} \;\;\; \forall x\in I.
  $$
For the function $h_{3}(z)=\frac{4(z+1)}{(z+0,5)^2},$ defined on the interval $[\frac{5}{6};3]$,  we have
   $$\frac{4q_{n-1}}{(q_n+0,5q_{n-1})^3}: \frac{1}{(q_n+0,5q_{n-1})(q_n+q_{n-1})}=$$
   $$=\frac{4q_{n-1}(q_n+q_{n-1})}{(q_n+0,5q_{n-1})^2}=
  \frac{4\left(\frac{q_n}{q_{n-1}}+1\right)}{\left(\frac{q_n}{q_{n-1}}+0,5\right)^2}=h_{3}(\frac{q_n}{q_{n-1}}) \leq h_{3}(\frac{5}{6})= \frac{33}{8}.$$
   Since $\frac{q_n}{q_{n-1}}\in[\frac{5}{6};3]$ and
   $$h_{3}^\prime(z)=4\left(\frac{(z+0,5)^2-2(z+1)(z+0,5)}{(z+0,5)^2}\right)=\frac{-4(z+1,5)}{(z+0,5)^3}<0 \;\; \forall z\in[\frac{5}{6};3],$$
   we obtain
   \[
    \left| f_n^{\prime\prime}(x)\right|\leq \frac{33}{8}\sum_{(\alpha_1,...,\alpha_n)\in A^n}\frac{1}{(q_n+0,5q_{n-1})(q_n+q_{n-1})}=\frac{33}{8} \;\;  \forall x\in I.\qedhere \]
\end{proof}
\begin{lemma}\label{lema5}
Let $(w_{n}(x))$ be a sequence of functions satisfying, for every natural $n$ and every $x \in I$,
$$
    w_{n+1}(x)=w_{n}\left( \frac{1}{0,5+x}\right)\cdot\frac{1}{(0,5+x)^2}+w_{n}\left( \frac{1}{1+x}\right)\cdot\frac{1}{(1+x)^2}.
$$
Then for all natural $n$, $m$ and all $x\in I$ 
  $$
  w_{n+m}(x)=\sum_{(\alpha_1,...,\alpha_n)\in A^n} w_{m}\left(\frac{p_n+xp_{n-1}}{q_n+xq_{n-1}} \right)\cdot\frac{1}{(q_n+xq_{n-1})^2}.
  $$
  \end{lemma}
  \begin{proof}
    The proof proceeds by induction on $n$. Since $ p_0=0, q_0=1, p_1=1,q_1=a_1$, the statement is true  when $n=1$.
    Let us assume the truth of the statement  when $n=k,$ i.e., the equality
    $$
    w_{k+m}(x)=\sum_{(\alpha_1,...,\alpha_k)\in A^k} w_{m}(x)\left(\frac{p_k+xp_{k-1}}{q_k+xq_{k-1}} \right)\cdot\frac{1}{(q_k+xq_{k-1})^2} \;\; \forall x\in I.
    $$
    Then
    \begin{align*}
      w_{k+1+m}(x)&=w_{k+m}\left( \frac{1}{0,5+x}\right)\cdot\frac{1}{(0,5+x)^2}+w_{k+m}\left( \frac{1}{1+x}\right)\cdot\frac{1}{(1+x)^2}=\\
      &=\frac{1}{(0,5+x)^2}\sum_{\stackrel{\alpha_i\in A}{i=\overline{1,k}}} w_{m} \left(\frac{p_k+\frac{1}{0,5+x}p_{k-1}}{q_k+\frac{1}{0,5+x}q_{k-1}}\right)\cdot
          \frac{1}{\left({q_k+\frac{1}{0,5+x}q_{k-1}}\right)^2}+\\
      &+\frac{1}{(1+x)^2}\sum_{\stackrel{\alpha_i\in A}{i=\overline{1,k}}}   w_{m}\left(\frac{p_k+\frac{1}{1+x}p_{k-1}}{q_k+\frac{1}{1+x}q_{k-1}}\right)\cdot
     \frac{1}{\left({q_k+\frac{1}{1+x}q_{k-1}}\right)^2}=\\
     &=\frac{1}{(0,5+x)^2}\sum_{\stackrel{\alpha_i\in A}{i=\overline{1,k}}} w_{m}\left(\frac{0,5p_k+p_{k-1}+xp_k}{0,5q_k+q_{k-1}+xq_k}\right)
             \cdot\frac{(0,5+x)^2}{\left({0,5q_k+q_{k-1}+xq_k}\right)^2}+\\
     &+\frac{1}{(1+x)^2}\sum_{\stackrel{\alpha_i\in A}{i=\overline{1,k}}}  w_{m}\left(\frac{p_k+p_{k-1}+xp_k}{q_k+q_{k-1}+xq_k}\right)
     \cdot\frac{(1+x)^2}{\left({q_k+q_{k-1}+xq_k}\right)^2}=\\
     &=\sum_{\stackrel{\alpha_i\in A}{i=\overline{1,n+1}}}w_{m} \left(\frac{p_{k+1}+xp_k}{q_{k+1}+xq_k}\right)\cdot\frac{1}{\left({q_{k+1}+xq_k}\right)^2}.
   \qedhere \end{align*}
      \end{proof}
\begin{lemma}\label{lema6}
   If for some non-negative $t,T$, and natural $k$ one has $$
   \frac{t}{(1+x)(2+x)}\leq f_{k}^{\prime}(x)\leq\frac{T}{(1+x)(2+x)}\;\;\forall x\in I,
   $$ then
  \begin{equation}\label{eq}
    \frac{t}{(1+x)(2+x)}\leq f_{k+1}^{\prime}(x)\leq\frac{T}{(1+x)(2+x)}\;\;\forall x\in I.
  \end{equation}
 \end{lemma}
 \begin{proof}
   We will prove the left side of inequality \eqref{eq}; the reasoning for the right side is similar. We have
   $$
   f_{k+1}'(x)=\frac{1}{(0,5+x)^2}\cdot f_k'\left(\frac{1}{0,5+x}\right)+\frac{1}{(1+x)^2}\cdot f_k'\left(\frac{1}{1+x}\right)\geq
   $$
   $$\geq\frac{1}{(0,5+x)^2}\cdot\frac{t}{\left(1+\frac{1}{0,5+x}\right)\left(2+\frac{1}{0,5+x}\right)}+
   \frac{1}{(1+x)^2}\cdot\frac{t}{\left(1+\frac{1}{1+x}\right)\left(2+\frac{1}{1+x}\right)}=
   $$
   $$
   =\frac{0,5t}{(1,5+x)(1+x)}+\frac{0,5t}{(1,5+x)(2+x)}=t\left(\frac{1}{1+x}-\frac{1}{1,5+x}+\frac{1}{1,5+x}-\frac{1}{2+x}\right)=$$
   \[= \frac{t}{(1+x)(2+x)}.\qedhere
   \]
 \end{proof}
\begin{lemma}\label{lema7}
  If for some positive constants $C_1$, $C_2$ and a natural number $m$ 
  $$\frac{C_1}{(x+1)(x+2)}\leq   f_m^{\prime} (x) \leq  \frac{C_2}{(x+1)(x+2)} \;\; \forall x \in I,    $$
 then for every natural $k$
  $$\frac{g(C_1)}{(x+1)(x+2)}\leq   f_{m+k}^{\prime}(x) \leq  \frac{w(C_2)}{(x+1)(x+2)} \;\; \forall x \in I,    $$
  where
$$g(t)=t+\frac{15}{4}\left(1-t\ln\frac{10}{9}\right)-\left(\frac{99}{4}+\frac{384}{225}t\right)\cdot\gamma_k,$$
$$
w(t)=t-\frac{15}{4}\left(t\ln\frac{10}{9}-1\right)+\left(\frac{99}{4}+\frac{384}{225}t\right)\cdot\gamma_{k},
$$
and
\begin{equation}\label{maz}
\gamma_{n}=\frac{1}{\frac{4}{17}((\frac{5+\sqrt{17}}{4})(\frac{1+\sqrt{17}}{4})^{n}-(\frac{5-\sqrt{17}}{4})(\frac{1-\sqrt{17}}{4})^{n})
 ((\frac{3+\sqrt{17}}{2})(\frac{1+\sqrt{17}}{4})^{n}-(\frac{3-\sqrt{17}}{2})(\frac{1-\sqrt{17}}{4})^{n})}.
\end{equation}
\end{lemma}
\begin{proof}
It is clear that
$$
f_{m}^\prime(x)\geq\frac{C_1}{(x+1)(x+2)} \quad \forall x\in I\Rightarrow \int\limits_{0,5}^{1}f_m^\prime(x)dx\geq\int\limits_{0,5}^{1}\frac{C_1dx}{(x+1)(x+2)}\Rightarrow C_1\leq\frac{1}{\ln\frac{10}{9}}.
$$
Arguing analogously with respect to $C_2$, we obtain
\begin{equation}\label{eqit}
C_1 \leq \frac{1}{\ln\frac{10}{9}} \leq C_2.
\end{equation}

From Lemma~\ref{lemma:basic} it follows that
    $$
    f_n(x)=1-f_n\left( \frac{1}{0,5+x}\right)+f_n\left(\frac{2}{3} \right)-f_n\left(\frac{1}{1+x} \right).
    $$
    Differentiating the last equality, we obtain
    $$
    f_{n+1}^\prime(x)=f_n^\prime\left( \frac{1}{0,5+x}\right)\cdot\frac{1}{(0,5+x)^2}+f_n^\prime\left( \frac{1}{1+x}\right)\cdot\frac{1}{(1+x)^2}.
    $$
    Since for $C \in R$ and for all $\forall x \in I$ the equality
 $$
    \frac{C}{(x+1)(x+2)}=\frac{C}{(\frac{1}{0,5+x}+1)
    (\frac{1}{0,5+x}+2)}\cdot\frac{1}{(0,5+x)^2}+\frac{C}
    {(\frac{1}{1+x}+1)(\frac{1}{1+x}+2)}
    \cdot\frac{1}{(1+x)^2},
$$
holds, it follows from Lemma \ref{lema6} that the sequence of functions
 $$\varphi_n(x)= f_n^\prime(x)-\frac{C_1}{(x+1)(x+2)}$$
  for arbitrary natural numbers $n$ and $m$ satisfies the condition
  $$
  \varphi_{m+n}(x)=\sum_{(\alpha_1,...,\alpha_n)\in A^n}\varphi_m\left(\frac{p_n+xp_{n-1}}{q_n+xq_{n-1}} \right)\cdot\frac{1}{(q_n+xq_{n-1})^2}.
  $$

 Denote $u_{n}(x)=\frac{p_n+xp_{n-1}}{q_n+xq_{n-1}}$.  Since the function $h_{4}(z)=\frac{z+0,5}{z+1}=1-\frac{0,5}{z+1}$
  is increasing on the interval $[\frac{5}{6};3]$, it follows that
 $$
   \frac{1}{(q_n+xq_{n-1})^2}:\frac{1}{(q_n+0,5q_{n-1})(q_n+q_{n-1})}=
   \frac{(q_n+0,5q_{n-1})(q_n+q_{n-1})}{(q_n+xq_{n-1})^2}\geq
  $$
  $$
  \geq\frac{q_n+0,5q_{n-1}}{q_n+q_{n-1}}=h_{4}\left(\frac{q_n}{q_{n-1}}\right)\geq\frac{8}{11}>\frac{1}{2} \;\; \forall x\in I,
  $$
   and consequently
  $$
\varphi_{m+n}(x)\geq  \sum_{(\alpha_1,...,\alpha_n)\in A^n}\varphi_m(u_n(x))\cdot\lambda(\Delta^{A_2}_{\alpha_1\alpha_2...\alpha_n}).
$$
It is known~\cite{prats}, that
$$\Delta^{A_2}_{\alpha_1\alpha_2...\alpha_n}=\left[\frac{p_{n}+p_{n-1}}{q_{n}+q_{n-1}};\frac{p_{n}+0,5p_{n-1}}{q_{n}+0,5q_{n-1}}\right]^{*},
     \lambda(\Delta^{A_2}_{{\alpha_1\alpha_2...\alpha_n}})=\frac{1}{(q_{n-1}+q_{n})(q_{n-1}+2q_{n})}, $$
$$q_{n}\geq \frac{2\sqrt{17}}{17}\left( \left(\frac{1+\sqrt{17}}{4}\right)^{n+1}-\left(\frac{1-\sqrt{17}}{4}\right)^{n+1}\right), $$
whence it readily follows that, in view of equality \eqref{maz} $\lambda(\Delta^{A_2}_{\alpha_1\alpha_2...\alpha_n}) \leq \gamma_{n}$.

We note that
$$\frac{p_{n}+p_{n-1}}{q_{n}+q_{n-1}}-\frac{p_{n}+xp_{n-1}}{q_{n}+xq_{n-1}}=\frac{(x-1)(p_{n}q_{n-1}-p_{n-1}q_{n})}{(q_{n}+q_{n-1})(q_{n}+xq_{n-1})}=
\frac{(x-1)(-1)^{n+1}}{(q_{n}+q_{n-1})(q_{n}+xq_{n-1})},$$

$$\frac{p_{n}+0,5p_{n-1}}{q_{n}+0,5q_{n-1}}-\frac{p_{n}+xp_{n-1}}{q_{n}+xq_{n-1}}=\frac{(x-0,5)(p_{n}q_{n-1}-p_{n-1}q_{n})}{(q_{n}+0,5q_{n-1})(q_{n}+xq_{n-1})}=
\frac{(x-0,5)(-1)^{n+1}}{(q_{n}+0,5q_{n-1})(q_{n}+xq_{n-1})},$$
than $u_{n}(x) \in \Delta^{A_2}_{\alpha_1\alpha_2...\alpha_n}$    for all $x \in I$.

By the mean value theorem, there exist points $u^{*}_{k}(x)\in\Delta^{A_2}_{\alpha_1(x)\alpha_2(x)...\alpha_k(x)}$ such that
$$
\int_{0,5}^{1}\varphi_m(z)dz=\sum\limits_{(\alpha_1,...,\alpha_k)\in A^k}\int\limits_{\Delta^{A_2}_{\alpha_1\alpha_2...\alpha_k}}\varphi_m(z)dz=
\sum\limits_{(\alpha_1,...,\alpha_k)\in A^k}\varphi_m(u_{k}^*(x))\cdot\lambda(\Delta^{A_2}_{\alpha_1\alpha_2...\alpha_k}).
$$
We obtain
$$
\varphi_{k+m}(x)-\int_{0,5}^{1}\varphi_m(z)dz\geq\sum\limits_{(\alpha_1,...,\alpha_k)\in A^k}(\varphi_m(u_{k}(x))-\varphi_m(u_{k}^{*}(x)))\lambda(\Delta^{A_2}_{\alpha_1\alpha_2...\alpha_k}).
$$
Taking into account Lemma \ref{lema3}, we have
$$
\left|\varphi_m(z)\right|\leq\max_{[0,5;1]}\frac{
\max\left(|\frac{1}{\ln\frac{10}{9}}-\frac{60}{7}|;|\frac{21}{2}-\frac{1}{\ln\frac{10}{9} }|\right)}{(x+1)(x+2)} \leq   \frac{1,0088}{1,5\cdot2,5}=\frac{1,0088}{3,75}\leq 0,2691.
$$

It is clear that the function $h_5(z)=\frac{1}{(z+1)^2}-\frac{1}{(z+2)^2}$ is decreasing on the interval $[0,5;1]$, since $h_5^\prime(z)=-\frac{2}{(z+1)^3}+\frac{2}{(z+2)^3}<0$ where $z\in I.$
Hence,
 $$\max_{[0,5;1]}\left|h_5(z)\right|=h_5(0,5)=\frac{1}{2,25}-\frac{1}{6,25}=\frac{64}{225}.$$

Taking into account Lagrange's theorem and Lemma \ref{lema4}, we obtain
$$
\left|\varphi_m(u_k(x))-\varphi_m(u_k^{*}(x))\right| \leq\left|f_m^\prime(u_k(x))-f_m^\prime(u_k^{*}(x))\right|+
$$
$$
+\left|\frac{C_1}{(u_k(x)+1)(u_k(x)+2)}-\frac{C_1}{(u_k^{*}(x)+1)(u_k^{*}(x)+2)}\right|\leq
$$
$$
\leq\left(\max_{t\in\Delta^{A_2}_{\alpha_1...\alpha_k}}f_m^{\prime\prime}(t)+
\max_{t\in\Delta^{A_2}_{\alpha_1...\alpha_k}}\left(\frac{C_1}{(z+1)(z+2)}\right)^\prime\Bigg|_{z=t}\right)
\cdot\lambda(\Delta^{A_2}_{\alpha_1...\alpha_k})\leq
$$
$$\leq    \frac{33}{8}\cdot\lambda(\Delta^{A_2}_{\alpha_1...\alpha_k})+C_1\cdot\max_{t\in\Delta^{A_2}_{\alpha_1...\alpha_k}}
\left|\frac{1}{(t+2)^2}-\frac{1}{(t+1)^2}\right|\cdot\lambda(\Delta^{A_2}_{\alpha_1...\alpha_k})
\leq\left(\frac{33}{8}+\frac{64C_1}{225}\right)
\cdot\lambda(\Delta^{A_2}_{\alpha_1...\alpha_k}).$$

We have
$$
\varphi_{m+k}(x)\geq\int_{0,5}^{1}\varphi_m(z)dz-\left(\frac{33}{8}+\frac{64}{225}C_1\right)\gamma_k,
$$
$$
f_{m+k}^\prime(x)\geq\frac{C_1}{(x+1)(x+2)}+\int_{0,5}^{1}\left(f_m^\prime(z)-\frac{C_1}{(z+1)(z+2)}\right)dz-\left(\frac{33}{8}+\frac{64}{225}C_1\right)\gamma_{k}=
$$
$$
=\frac{C_1}{(x+1)(x+2)}+\left(1-\ln\left(\frac{10}{9}\right)C_1\right)-\left(\frac{33}{8}+\frac{64}{225}C_1\right)\gamma_k\geq\frac{g(C_1)}{(x+1)(x+2)},
$$

Similarly, consider the function $\psi_n(x)=\frac{C_2}{(x+1)(x+2)}-f_n^\prime(x)$
and obtain
$$
\psi_{m+k}(x)\geq\int_{0,5}^{1}\psi_m(z)dz-\left(\frac{33}{8}+\frac{64}{225}C_2\right)\gamma_{k},
$$
\[
f_{m+k}^\prime(x)\leq\frac{C_2}{(x+1)(x+2)}-\left(\ln \left(\frac{10}{9}\right)C_2-1\right)+\left(\frac{33}{8}+\frac{64}{225}C_2\right)\gamma_k\leq\frac{w(C_2)}{(x+1)(x+2)}.
\qedhere\]
\end{proof}
\begin{theorem}
For every $n\in N$ and $x\in I$, the following condition holds
\begin{equation}\label{mr}
  -\ln\frac{10}{9}|r_{[\sqrt{n}]}|\leq f_n(x)-\ln^{-1}\left(\frac{10}{9}\right)\ln\frac{5(x+1)}{3(x+2)}
\leq\ln\frac{10}{9}|d_{[\sqrt{n}]}|,
\end{equation}
where
$$
r_{n}=-\ln^{-1}\left(\frac{10}{9}\right)+\left(\frac{60}{7}-\frac{\frac{15}{4}-\frac{99}{4}\gamma_n}
{\frac{15}{4}\ln\frac{10}{9}+\frac{384}{225}\gamma_n}\right)\cdot \left(1-\frac{15}{4}\ln\frac{10}{9}-\frac{384}{225}\gamma_n\right)^{n-1}+
\frac{\frac{15}{4}-\frac{99}{4}\gamma_n}{\frac{15}{4}\ln\frac{10}{9}+\frac{384}{225}\gamma_n},
$$
$$
d_{n}=-\ln^{-1}\left(\frac{10}{9}\right)+\left(\frac{21}{2}-\frac{\frac{99}{4}\gamma_n+\frac{15}{4}}
{\frac{15}{4}\ln\frac{10}{9}-\frac{384}{225}\gamma_n}\right)\cdot \left(1-\frac{15}{4}\ln\frac{10}{9}+\frac{384}{225}\gamma_n\right)^{n-1}+
 \frac{\frac{99}{4}\gamma_{n}+\frac{15}{4}}{\frac{15}{4}\ln\frac{10}{9}-\frac{384}{225}\gamma_n},
$$
the sequence $(\gamma_{n})$ is defined by equality \eqref{maz},$[n]$ denotes the integer part of $n$.
\end{theorem}
\begin{proof}
Denote $g^{(n)}(x)=\underbrace{g(g(\ldots g(x)))}_{n},\quad w^{(n)}(x)=\underbrace{w(w(\ldots w(x)))}_{n}.
$
For a given natural number $t$, in Lemma~\ref{lema7} we take the pairs $(m,k)$ to be $(t,t)$,$(2t,t)$,\ldots,$((t-1)t,t).$

Taking into account Lemma \ref{lema3} and condition \eqref{eqit}, we successively obtain
$$
f_{2t}^\prime(x)\geq\frac{g(\frac{60}{7})}{(x+1)(x+2)} \quad \forall x\in I,
$$
$$
f_{3t}^\prime(x)\geq\frac{g^{(2)}(\frac{60}{7})}{(x+1)(x+2)} \quad \forall x\in I,
$$
$$
\ldots \quad \ldots \quad \ldots
$$
$$
f_{t^2}^\prime(x)\geq\frac{g^{(t-1)}(\frac{60}{7})}{(x+1)(x+2)} \quad \forall x\in I.
$$

It is easy to see that for the recursive sequence $x_{n+1}=a\cdot x_n+b$, with $a\neq1$, the following equality holds 
$$
x_n=\left(x_0-\frac{b}{1-a}\right)a^n+\frac{b}{1-a}.
$$
Taking $x_{0}=\frac{60}{7}, \quad a=1-\frac{15}{4}\ln\frac{10}{9}-\frac{384}{225}\gamma_t,  \quad b=\frac{15}{4}-\frac{99}{4}\gamma_t$, we obtain
$$
g^{(t-1)}(\frac{60}{7})=\left(\frac{60}{7}-\frac{\frac{15}{4}-\frac{99}{4}\gamma_t}{\frac{15}{4}
\ln\frac{10}{9}+\frac{384}{225}\gamma_t}\right)\cdot \left(1-\frac{15}{4}\ln\frac{10}{9}-\frac{384}{225}\gamma_t\right)^{t-1}+
\frac{\frac{15}{4}-\frac{99}{4}\gamma_t}{\frac{15}{4}\ln\frac{10}{9}+\frac{384}{225}\gamma_t}.
$$
Similarly,
$$
f_{t^2}^\prime(x)\leq\frac{w^{(t-1)}(\frac{21}{2})}{(x+1)(x+2)} \quad \forall x\in I,
$$
where
$$
w^{(t-1)}(\frac{21}{2})=\left(\frac{21}{2}-\frac{\frac{99}{4}\gamma_t+
\frac{15}{4}}{\frac{15}{4}\ln\frac{10}{9}-\frac{384}{225}\gamma_t}\right)\cdot \left(1-\frac{15}{4}\ln\frac{10}{9}+\frac{384}{225}\gamma_t\right)^{t-1}+
 \frac{\frac{99}{4}\gamma_{t}+\frac{15}{4}}{\frac{15}{4}\ln\frac{10}{9}-\frac{384}{225}\gamma_t}.
$$
Thus, for every natural number $t$
$$
\frac{\ln^{-1}\left(\frac{10}{9}\right)}{(x+1)(x+2)}+\frac{r_t}{(x+1)(x+2)}\leq f_{t^2}^\prime(x)\leq\frac{\ln^{-1}\left(\frac{10}{9}\right)}{(x+1)(x+2)}+\frac{d_t}{(x+1)(x+2)}.
$$

Taking $t=[\sqrt{n} ]$ and using Lemma~\eqref{lema6}, together with the fact that $[\sqrt{n}]^2\leq(\sqrt{n})^2=n$ we obtain
$$
\frac{\ln^{-1}\left(\frac{10}{9}\right)}{(x+1)(x+2)}+\frac{r_{[\sqrt{n}]}}{(x+1)(x+2)}\leq f_{n}^\prime(x)\leq\frac{\ln^{-1}\left(\frac{10}{9}\right)}{(x+1)(x+2)}+\frac{d_{[\sqrt{n}]}}{(x+1)(x+2)}.
$$
Denoting
$$f(x)=\frac{\ln\left(x+1\right)-\ln\left(x+2\right)+\ln\frac{5}{3}}{\ln\frac{10}{9}}=
\ln^{-1}\left(\frac{10}{9}\right)\ln\frac{5(x+1)}{3(x+2)}$$
and integrating the last inequality over the interval $[0,5;\tau]$ ($\tau \in I$), we obtain
 $$
\ln\left(\frac{10}{9}\right)r_{[\sqrt{n}]}f(\tau)\leq f_n(\tau)-\frac{\ln\left(\tau+1\right)-\ln\left(\tau+2\right)+\ln\frac{5}{3}}{\ln\frac{10}{9}}\leq\ln\left(\frac{10}{9}\right)
d_{[\sqrt{n}]}f(\tau),
$$
whence, taking into account that $0\leq f(x)\leq 1$ for  $x\in I$, equality~\eqref{mr} follows, which completes the proof.
\end{proof}

Remark that
$$\left(\frac{9+\sqrt{17}}{8}\right)^{-1} \approx 0,6599; \;\;  1-\frac{15}{4}\ln \frac{10}{9} \approx 0,6049; \;\; \ln \frac{10}{9}  \approx 0,1054,  $$
$$\frac{\frac{99}{4}ln\frac{10}{9}+ \frac{384}{225}}{ \frac{15}{4}\ln\frac{10}{9}}  \approx 10,9196;  \;\; \gamma_{n}\sim      \frac{34}{(5+\sqrt{17})(3+\sqrt{17})}     \left(\frac{9+\sqrt{17}}{8}\right)^{-n},    $$
$$
|r_n|\sim      \frac{\frac{99}{4}\ln\frac{10}{9}+ \frac{384}{225}}{ \frac{15}{4}\ln \frac{10}{9}}      \left(\frac{9+\sqrt{17}}{8}\right)^{-n}\sim
|d_n|.
$$
\begin{corollary}
There exists a natural number $M$ such that for every natural number $n\geq M$
$$
 \left|f_{n}(x)-\ln^{-1}\left(\frac{10}{9}\right)\ln\frac{5(x+1)}{3(x+2)}\right|\leq 10,9197 \cdot 0,66^{\sqrt{n}}, \forall x \in I.
$$
\end{corollary}

\textit{Conclusion.} In conclusion, we note that a subject of further research is to obtain estimates for the rate of convergence of the sequence $f_n(x)$ to the limiting function $f(x)$.

\end{document}